\documentclass[11pt]{article}

\usepackage{enumerate, amsmath, amsthm, amsfonts, amssymb, xy}
\usepackage[margin=1in]{geometry}
\usepackage{hyperref}
\usepackage{algorithm, algorithmic}

\input xy
\xyoption{all}

\renewcommand{\algorithmicrequire}{\textbf{Input:}}
\renewcommand{\algorithmicensure}{\textbf{Output:}}

\newtheorem{theorem}{Theorem}[section]
\newtheorem{proposition}[theorem]{Proposition}
\newtheorem{lemma}[theorem]{Lemma}
\newtheorem{definition}[theorem]{Definition}
\newtheorem{example}[theorem]{Example}
\newtheorem{rmk}[theorem]{Remark}
\newtheorem{corollary}[theorem]{Corollary}
\newtheorem{conjecture}[theorem]{Conjecture}

\newenvironment{eg}{\begin{example}\rm}{\end{example}}
\newenvironment{remark}{\begin{rmk}\rm}{\end{rmk}}
\numberwithin{equation}{section}

\newcommand{\N}{\mathbf{N}}
\newcommand{\Q}{\mathbf{Q}}
\newcommand{\R}{\mathbf{R}}
\newcommand{\Z}{\mathbf{Z}}

\renewcommand{\phi}{\varphi}

\newcommand{\ul}[1]{\underline{#1}}

\makeatletter
\def\Ddots{\mathinner{\mkern1mu\raise\p@
\vbox{\kern7\p@\hbox{.}}\mkern2mu
\raise4\p@\hbox{.}\mkern2mu\raise7\p@\hbox{.}\mkern1mu}}
\makeatother

\newcommand{\QP}{\mathbf{QP}_{\gg 0}}
\DeclareMathOperator{\ggcd}{ggcd}

\DeclareMathOperator{\lcm}{lcm}
\newcommand{\lc}{\operatorname{lc}}

\newcommand{\conv}{\operatorname{conv}}

\title{Generalized Ehrhart polynomials}
\date{October 29, 2010}
\author{Sheng Chen, Nan Li, and Steven V Sam}

\begin{document}

\maketitle
\begin{abstract}

  Let $P$ be a polytope with rational vertices. A classical theorem of
  Ehrhart states that the number of lattice points in the dilations
  $P(n) = nP$ is a quasi-polynomial in $n$. We generalize this theorem
  by allowing the vertices of $P(n)$ to be arbitrary rational
  functions in $n$. In this case we prove that the number of lattice
  points in $P(n)$ is a quasi-polynomial for $n$ sufficiently
  large. Our work was motivated by a conjecture of Ehrhart on the
  number of solutions to parametrized linear Diophantine equations
  whose coefficients are polynomials in $n$, and we explain how these
  two problems are related.

\end{abstract}

\section{Introduction.}

In this article, we relate two problems, one from classical number
theory, and one from lattice point enumeration in convex bodies.
Motivated by a conjecture of Ehrhart \cite{ehrhart} and a result of Xu
\cite{xu}, we study linear systems of Diophantine equations with a
single parameter. To be more precise, we suppose that the coefficients
of our system are given by polynomial functions in a variable $n$, and
also that the number of solutions $f(n)$ in nonnegative integers for
any given value of $n$ is finite. We are interested in the behavior of
the function $f(n)$, and in particular, we prove that $f(n)$ is
\textbf{eventually a quasi polynomial}, i.e., there exists some period
$s$ and polynomials $f_i(t)$ for $i=0,\dots,s-1$ such that for $t \gg
0$, the number of solutions for $n \equiv i \pmod s$ is given by
$f_i(n)$. The other side of our problem can be stated in a similar
fashion: suppose that $P(n)$ is a convex polytope whose vertices are
given by rational functions in $n$. Then the number of integer points
inside of $P(n)$, as a function of $n$, enjoys the same properties as
that of $f$ as above. We now describe in more detail some examples and
the statements of our results.

\subsection{Diophantine equations.}

As a warmup to our result, we begin with two examples. The first is a
result of Popoviciu. Let $a$ and $b$ be relatively prime positive
integers. We wish to find a formula for the number of nonnegative
integer solutions $(x,y)$ to the equation $ax + by = n$. For a real
number $x$, let $\lfloor x \rfloor$ denote the greatest integer less
than or equal to $x$, and define $\{ x \} = x - \lfloor x \rfloor$ to
be the fractional part of $x$. Then the number of such solutions is
given by the formula
\begin{align} \label{popoviciu}
\frac{n}{ab} - \left\{ \frac{na^{-1}}{b} \right\} - \left\{
  \frac{nb^{-1}}{a} \right\} + 1,
\end{align}
where $a^{-1}$ and $b^{-1}$ satisfy $aa^{-1} \equiv 1 \pmod b$ and
$bb^{-1} \equiv 1 \pmod a$. See \cite[Chapter 1]{ccd} for a proof. In
particular, this function is a quasi-polynomial in $n$.

For the second example which is a generalization of the first example,
consider the number of solutions $(x,y,z) \in \Z_{\ge 0}^3$ to the
matrix equation
\begin{align} \label{matrixexample}
\left( \begin{matrix} x_1 & x_2 & x_3 \\ y_1 & y_2 & y_3 \end{matrix}
\right) \left( \begin{matrix} x \\ y \\ z \end{matrix} \right) =
\left( \begin{matrix} m_1 \\ m_2 \end{matrix} \right)
\end{align}
where the $x_i$ and $y_i$ are fixed positive integers and $x_{i+1}y_i
< x_iy_{i+1}$ for $i=1,2$. Write $Y_{ij} = x_iy_j - x_jy_i$. We assume
that $\gcd(Y_{12}, Y_{13}, Y_{23}) = 1$, so that there exist integers
(not unique) $f_{ij}, g_{ij}$ such that
\[
\gcd(f_{12}Y_{13} + g_{12}Y_{23}, Y_{12}) = 1, \quad \gcd(f_{13}Y_{12}
+ g_{13}Y_{23}, Y_{13}) = 1, \quad \gcd(f_{23}Y_{13} + g_{23}Y_{12},
Y_{23}) = 1.
\]
Now define two regions $\Omega_i = \{(x,y) \mid \frac{y_i}{x_i} <
\frac{y}{x} < \frac{y_{i+1}}{x_{i+1}} \}$ for $i=1,2$. Then if $m =
(m_1, m_2) \in \Z^2$ is in the positive span of the columns of the
matrix in \eqref{matrixexample}, there exist Popoviciu-like formulas
for the number of solutions of \eqref{matrixexample} which depend only
on whether $m \in \Omega_1$ or $m \in \Omega_2$, and the numbers
$Y_{ij}, f_{ij}, g_{ij}, x_i, y_i$. See Section~\ref{section:example1} for the
precise statement.

In particular, one can replace the $x_i$, $y_i$, and $m_i$ by
polynomials in $n$ in such a way that for all values of $n$, the
condition $\gcd(Y_{12}, Y_{13}, Y_{23}) = 1$ holds. For a concrete
example, consider the system
\[
\left( \begin{matrix} 2n+1 & 3n+1 & n^2 \\ 2 & 3 & n+1 \end{matrix}
\right) \left( \begin{matrix} x \\ y \\ z \end{matrix} \right) =
\left( \begin{matrix} 3n^3 + 1 \\ 3n^2 + n - 1 \end{matrix} \right).
\]
Then for $n \gg 0$, we have that
\[
\frac{3}{3n+1} < \frac{3n^2 + n - 1}{3n^3 + 1} < \frac{n+1}{n^2},
\]
so that for these values of $n$, there exists a quasi-polynomial that
counts the number of solutions $(x,y,z)$.

Given these examples, we are ready to state our general theorem. We
denote by $\QP$ the set of functions $f \colon \Z \to \Z$ which are
eventually quasi-polynomial.

\begin{theorem} \label{ehrhartconj} Let $A(n)$ be an $m \times k$
  matrix, and $b(n)$ be a column vector of length $m$, such that their
  entries are integer coefficient polynomials in $n$. If $f(n)$
  denotes the number of nonnegative integer vectors $x$ satisfying
  $A(n)x = b(n)$ (assuming that these values are finite), then $f \in
  \QP$.
\end{theorem}

This theorem generalizes the conjecture \cite[Exercise
4.12]{stanley}. See \cite[p. 139]{ehrhart} for some verified cases of
a conjectural multivariable analogue, which we state here. Let $S
\subset \Z^r$ be some subset. We say that a function $f \colon S \to
\Z$ is a {\bf multivariate quasi-polynomial} if there exists a finite
index sublattice $L \subset \Z^r$ such that $f$ is a polynomial
function on each coset of $L$ intersected with $S$.

\begin{conjecture}[Ehrhart] Let $A(n_1, \dots, n_r)$ be an $m \times
  k$ matrix and $b(n_1, \dots, n_r)$ be a column vector of length $m$,
  such that all entries are linear functions in $n_1, \dots, n_r$ with
  integer coefficients such that for all $n_1, \dots, n_r$, the number
  of nonnegative integer solutions $x$ to $A(n_1, \dots, n_r)x =
  b(n_1, \dots, n_r)$ is finite. Then there exist finitely many
  polyhedral regions $R_1, \dots, R_N$ covering $\R^r_{\ge 0}$ such
  that $f$ is a multivariate quasi-polynomial when restricted to each
  $R_i$.
\end{conjecture}

\subsection{Lattice point enumeration.}
We first recall a classical theorem due to Pick. Let $P \subset \R^2$
be a convex polygon with integral vertices. If $A(P)$, $I(P)$, and
$B(P)$ denote the area of $P$, the number of integer points in the
interior of $P$, and the number of integer points on the boundary of
$P$, respectively, then one has the equation
\[
A(P) = I(P) + \frac{1}{2}B(P) - 1.
\]
Now let us examine what happens with dilates of $P$: define $nP = \{nx
\mid x \in P\}$. Then of course $A(nP) = A(P)n^2$ and $B(nP) = nB(P)$
whenever $n$ is a positive integer, so we can write
\[
A(P)n^2 = I(nP) + \frac{1}{2}B(P)n - 1,
\]
or equivalently,
\[
\#(nP \cap \Z^2) = I(nP) + B(nP) = A(P)n^2 + \frac{1}{2}B(P)n + 1,
\]
which is a polynomial in $n$. The following theorem of Ehrhart says
that this is always the case independent of the dimension, and we can
even relax the integral vertex condition to rational vertices:

\begin{theorem}[Ehrhart] \label{ehrharttheorem}
  Let $P \subset \R^d$ be a polytope with rational vertices. Then the
  function $L_P(n) = \#(nP \cap \Z^d)$ is a quasi-polynomial of degree
  $\dim P$. Furthermore, if $D$ is an integer such that $DP$ has
  integral vertices, then $D$ is a period of $L_P(n)$. In particular,
  if $P$ has integral vertices, then $L_P(n)$ is a polynomial.
\end{theorem}

\begin{proof} See \cite[Theorem 4.6.25]{stanley} or \cite[Theorem
  3.23]{ccd}. \end{proof}

The function $L_P(t)$ is called the {\bf Ehrhart quasi-polynomial} of
$P$. One can see this as saying that if the vertices of $P$ are $v_i =
(v_{i1}, \dots, v_{id})$, then the vertices of $nP$ are given by the
linear functions $v_i(n) = (v_{i1}n, \dots, v_{id}n)$. We generalize
this as

\begin{theorem}\label{polytope}
  Given polynomials $v_{ij}(x), w_{ij}(x) \in \Z[x]$ for $0 \le i \le
  s$ and $1 \le j \le d$, let $n$ be a positive integer such that
  $w_{ij}(n) \ne 0$ for all $i,j$. This is satisfied by $n$
  sufficiently large, so we can define a rational polytope $P(n) =
  \conv(p^0(n), p^1(n), \dots, p^s(n)) \in \R^d$, where $p^i(n) =
  (\frac{v_{i1}(n)}{w_{i1}(n)}, \dots,
  \frac{v_{id}(n)}{w_{id}(n)})$. Then $\#(P(n)\cap\Z^d)\in\QP$.
\end{theorem}

We call the function $\#(P(n) \cap \Z^d)$ a {\bf generalized Ehrhart
  polynomial}.

In Section~\ref{equi}, we explain the equivalence of the two problems
just mentioned, then prove our main
result Theorem~\ref{ehrhartconj} in
Section~\ref{section:mainproof}. The proof gives us an algorithm to compute these generalized Ehrhart polynomials, but it could be very complicated in practice. For computational reasons, we introduce the notions of
generalized division and generalized gcd for the ring $\Z[x]$ in Section~\ref{generalizeddivision}. In a
sense, these generalize the usual notions of division and gcd for the
ring of integers $\Z$. The methods and algorithms are quite similar,
but are more involved due to technical complications. The proofs for the
correctness of the algorithms are given in this section. As an application of these tools, in Section~\ref{section:application}, we will describe explicit computations of some special generalized Ehrhart polynomials.

\subsection*{Acknowledgements.} We thank Richard Stanley for useful
discussions and for reading previous drafts of this article. We also
thank an anonymous referee for helping to improve the quality and
readability of the paper. Sheng Chen was sponsored by Project 11001064
supported by National Natural Science Foundation of China. Steven Sam
was supported by an NSF graduate fellowship and an NDSEG fellowship.

\section{Equivalence of the two problems}\label{equi}

As we shall see, the two problems of the Diophantine equations and
lattice point enumeration are closely intertwined. In this section, we
want to show that Theorem \ref{ehrhartconj} is equivalent to Theorem
\ref{polytope}. Before this, let us see the equivalence of Theorem
\ref{polytope} with the following result. For notation, if $x$ and $y$
are vectors, then $x \ge y$ if $x_i \ge y_i$ for all $i$.

\begin{theorem}\label{ineq}
  For $n\gg 0$, define a rational polytope $P(n)=\{x \in \R^d \mid
  V(n)x \ge c(n)\}$, where $V$ is an $r \times d$ matrix, and $c$ is
  an $r \times 1$ column vector, both of whose entries are integer
  coefficient polynomials. Then $\#(P(n)\cap\Z^d)\in\QP$.
\end{theorem}

Notice that the difference of Theorem \ref{polytope} and Theorem
\ref{ineq} is that one defines a polytope by its vertices and the
other by hyperplanes. So we will show their equivalence by presenting
a generalized version of the algorithm connecting ``vertex
description'' and ``hyperplane description'' of a polytope.

The connection is based on the fact that we can compare two rational
functions $f(n)$ and $g(n)$ when $n$ is sufficiently large. For
example, if $f(n)=n^2-4n+1$ and $g(n)=5n$, then $f(n)>g(n)$ for all
$n>9$, we denote this by $f(n)>_{\text{even}}g(n)$ (``even'' being
shorthand for ``eventually''). Therefore, given a point and a
hyperplane, we can test their relative position. To be precise, let
$p(n)=(r_1(n),\dots,r_k(n))$ be a point where the $r_i(n)$ are
rational functions and let $F(x,n) = a_1(n)x_1 + a_2(n)x_2 + \dots +
a_k(n)x_k = 0$ be a hyperplane where all the $a_i(n)$ are polynomials
of $n$. Then exactly one of the following will be true:
\[
F(p,n)=_{\text{even}}0; \quad F(p,n)>_{\text{even}}0; \quad
F(p,n)<_{\text{even}}0.
\]
Given this, we can make the following definition. We say that two
points $p(n)$ and $q(n)$ {\bf lie} (resp., {\bf weakly lie}) {\bf on
  the same side of $F(p,n)$} if $F(p,n) F(q,n) >_{\text{even}} 0$
(resp., $F(p,n)F(q,n)\ge_{\text{even}}0$).

\subsection{Equivalence of Theorem~\ref{polytope} and
  Theorem~\ref{ineq}.}

Going from the ``vertex description'' to the ``hyperplane
description'':

Given all vertices of a polytope $P(n)$, whose coordinates are all
rational functions of $n$, we want to get its ``hyperplane
description'' for $n\gg 0$. Let $F(x,n)$ be a hyperplane defined by a
subset of vertices.  If all vertices lie weakly on one side of
$F(x,n)$, we will keep it together with $\ge 0$, or $\le 0$ or $=0$
indicating the relative position of this hyperplane and the
polytope. We can get all the hyperplanes defining the polytope by this
procedure.

~

Going from the ``hyperplane description'' to the ``vertex
description'':

Let $P(n)=\{x \in \R^d \mid V(n)x \ge c(n)\}$ be a polytope, where $V$
is an $r \times d$ matrix, and $c$ is an $r \times 1$ column vector,
both of whose entries are integer coefficient polynomials. We want to
find its vertex description. Let $f_1(n), \dots, f_r(n)$ be the linear
functionals defined by the rows of $V(n)$. So we can rewrite $P(n)$ as
\[
P(n) = \{x \in \R^d \mid \langle f_i(n), x \rangle \ge c_i(n) \text{
  for all } i \}.
\]
The vertices of $P(n)$ can be obtained as follows. For every
$d$-subset $I \subseteq \{1, \dots, r\}$, if the equations
$\{\langle f_i(n), x \rangle = c_i(n)\mid i \in I\}$ are linearly
independent for $n\gg 0$, and their intersection is nonempty, then
it consists of a single point, which we denote by $v_I(n)$. If
$\langle f_j(n), v_I(n) \rangle \ge c_j(n)$ for all $j$, then
$v_I(n)\in \Q(n)^d$ is a vertex of $P(n)$, and all vertices are
obtained in this way. We claim that the subsets $I$ for which
$v_I(n)$ is a vertex remains constant if we take $n$ sufficiently
large. First, the notion of being linearly independent equations can
be tested by showing that at least one of the $d \times d$ minors of
the rows of $V(n)$ indexed by $I$ does not vanish. Since these
minors are all polynomial functions, they can only have finitely
many roots unless they are identically zero. Hence taking $n \gg 0$,
we can assume that $\{ f_i(n) \mid i \in I \}$ is either always
linearly dependent or always linearly independent. Similarly, the
sign of $\langle f_j(n), v_I(n) \rangle$ is determined by the sign
of a polynomial, and hence is constant for $n \gg 0$.

\subsection{Equivalence of Theorem~\ref{ehrhartconj} and
  Theorem~\ref{ineq}.}

We can easily transform an inequality to an equality by introducing
some slack variables and we can also represent an equality $f(n,x)=0$
by two inequalities $f(n,x)\ge 0$ and $-f(n,x)\ge 0$. So the main
difference between the two theorems is that Theorem~\ref{ehrhartconj}
is counting nonnegative solutions while Theorem~\ref{ineq} is counting
all integral solutions. But we can deal with this by adding constraints
on each variable.

A more interesting connection between Theorem~\ref{ehrhartconj} and
Theorem~\ref{ineq} is worth mentioning here. First consider any fixed
integer $n$. Then the entries of $A(n)$ and $b(n)$ in the linear
Diophantine equations $A(n)x = b(n)$ of Theorem~\ref{ehrhartconj} all
become integers. For an integer matrix, we can calculate its Smith
normal form. Similarly, we can use a generalized Smith normal form for
matrices over $\QP$ to get a transformation from
Theorem~\ref{ehrhartconj} to Theorem~\ref{ineq}.

Then given $A(n)$ and $b(n)$, by Theorem~\ref{smith}, we can put
$A(n)$ into generalized Smith normal form: $D(n) = U(n)A(n)V(n)$ for
some matrix
\[
D(n) = (\text{diag}(d_1(n), \dots, d_r(n), 0, \dots, 0)|\textbf{0})
\]
with nonzero entries only on its main diagonal, and unimodular
matrices $U(n)$ and $V(n)$. Then the equation $A(n)x = b(n)$ can be
rewritten as $D(n)V(n)^{-1}x = U(n)b(n)$. Set $y = V(n)^{-1}x$ and
$b'(n) = U(n)b(n)$. By the form of $D(n)$, we have a solution $y$ if
and only if $d_i(n)$ divides $b'_i(n)$ for $i=1,\dots,r$, and for any
given solution, the values $y_{r+1}, \dots, y_k$ can be
arbitrary. However, since $V(n)y = x$, we need to require that $V(n)y
\ge 0$, and any such $y$ gives a nonnegative solution $x$ to the
original problem. Simplifying $V(n)y \ge 0$, where
$V(n)=(v_1(n),\dots, v_k(n))$, we get $V'(n)X\ge c(n)$, where $V'(n) =
(v_{r+1}(n),\dots,v_k(n))$, $X=(y_{r+1},\dots, y_k)$ and
$c(n)=-(v_1(n)y_1+\cdots +v_r(n)y_r)$. Although $V'(n)$ and $c(n)$ has
entries in $\QP$, we can assume that they are polynomials by dealing
with each constituent of the quasi-polynomials separately. So we
reduce Theorem~\ref{ehrhartconj} to Theorem~\ref{ineq}.

\begin{eg}
Consider the nonnegative integer solutions to 
\begin{align} \label{eqn:smitheg}
\begin{pmatrix}
n^2+2n & 2n+2
 \end{pmatrix}
\begin{pmatrix}
  x_1 \\
  x_2 \end{pmatrix}=(2n+4)(n^2+2n).
\end{align}
Write $A(n) = \begin{pmatrix} n^2 + 2n & 2n+2 \end{pmatrix}$. When
$n=2m$, the Smith normal form of $A(2m)$ is
$$
\begin{pmatrix}
4m^2+4m & 4m+4
 \end{pmatrix}\begin{pmatrix}
1 & -1 \\ 
-(m-1) & m
\end{pmatrix}
=
\begin{pmatrix}
4m+4 & 0
 \end{pmatrix}.$$

So the equation \eqref{eqn:smitheg} becomes 
$$
\begin{pmatrix}
4m+4 & 0
 \end{pmatrix}\begin{pmatrix}
1 & -1 \\ 
-(m-1) & m
      \end{pmatrix}
^{-1}
\begin{pmatrix}
x_1 \\ 
x_2
      \end{pmatrix}=(2n+4)(n^2+2n).$$
Set $\begin{pmatrix}
1 & -1 \\ 
-(m-1) & m
      \end{pmatrix}
^{-1}
\begin{pmatrix}
x_1 \\ 
x_2
      \end{pmatrix}=
\begin{pmatrix}
y_1 \\ 
y_2
      \end{pmatrix}$, so that we have
$$
\begin{pmatrix}
4m+4 & 0
 \end{pmatrix}
\begin{pmatrix}
y_1 \\ 
y_2
      \end{pmatrix}=(2n+4)(n^2+2n),$$
and thus $y_1=n^2+2n$. By the condition 
$
\begin{pmatrix}
x_1 \\ 
x_2
      \end{pmatrix}\ge 0$, we require 
$$\begin{pmatrix}
1 & -1 \\ 
-(m-1) & m
      \end{pmatrix}
\begin{pmatrix}
n^2+2n \\ 
y_2
      \end{pmatrix}\ge 0$$ which gives us
$$\frac{(m-1)(n^2+2n)}{m}\le y_2 \le n^2+2n, $$ a one dimensional
polytope. So the number of solutions for $n=2m$ is $f(n)=2n+5$ 

We can do the case $n=2m+1$ similarly. The Smith normal form of
$A(2m+1)$ is 
$$
\begin{pmatrix}
4m^2+8m+2 & 4m+6
 \end{pmatrix}
\begin{pmatrix}
1 & -2 \\ 
-m & 2m+1
      \end{pmatrix}
=
\begin{pmatrix}
2m+3 & 0
 \end{pmatrix}$$ and 
$\frac{2m(n^2+2n)}{2m+1}\le y_2 \le n^2+2n$, thus $f(n)=n+3$.
\end{eg}

The proof of Theorem~\ref{smith} is based on a theory of generalized
division and GCD over the ring $\Z[x]$, which mainly says that for
$f(x),g(x) \in \Z[x]$, the functions $\left\lfloor\frac{f(n)}{g(n)}
\right\rfloor$, $\left\{\frac{f(n)}{g(n)} \right\}$, and
$\gcd(f(n),g(n))$ lie in the ring $\QP$. One interesting consequence
of these results is that every finitely generated ideal in $\QP$ is
principal, despite the fact that $\QP$ is not Noetherian. We developed
this theory in order to appoach Theorem~\ref{ehrhartconj} at first,
but subsequently have found a proof that circumvents its use. Further
details can be found in Section~\ref{generalizeddivision}.

\section{Proof of Theorem~\ref{ehrhartconj}} \label{section:mainproof}
To prove Theorem \ref{ehrhartconj}, we will use a ``writing in base
$n$'' trick, to reduce equations with polynomial coefficients to
linear functions. Briefly, the idea of the following ``writing in base
$n$'' trick is as follows: given a linear Diophantine equation
$$a_1(n)x_1+a_2(n)x_2+\cdots+a_k(n)x_k=m(n)$$
with polynomial coefficients $a_i(n)$ and $m(n)$, fix an integer $n$,
then the coefficients all become integers. Now consider a solution
$(x_1,x_2,\dots,x_k)$ with $x_i\in \Z_{\ge 0}$. Put the values of
$(x_1,x_2,\dots,x_k)$ into the equation, then both sides become an
integer. Then we use the fact that any integer has a unique
representation in base $n$ ($n$ is a fixed number), and compare the
coefficient of each power of $n$ in both sides of the equation.

Finally, letting $n$ change, we happen to have a uniform expression
for both sides in base $n$ when $n$ is sufficiently large. Moreover,
the coefficient of each power of $n$ in both sides of the equation are
all linear functions of $n$ (Lemma \ref{repr} and Lemma
\ref{reduce}). Then by Lemma \ref{linear}, we can reduce these
equations with linear function coefficients to the case when we can
apply Ehrhart's theorem (Theorem \ref{ehrharttheorem}) to show that
the number of solutions are quasi-polynomials of $n$. This completes
the proof of Theorem \ref{ehrhartconj}.

\begin{lemma}\label{repr}
  Given $p(x)\in \Z[x]$ with $p(n)>0$ for $n\gg 0$ (i.e., $p(x)$ has
  positive leading coefficient), there is a unique representation of
  $p(n)$ in base $n$:
  \[
  p(n)=c_d(n) n^d + \dots + c_1(n) n + c_0(n),
  \]
  where $c_i(n)$ is a linear function of $n$ such that for $n \gg 0$,
  $0 \le c_i(n) \le n-1$ for $i=0,1, \dots, d$ and $0<c_d(n)\le
  n-1$. We denote $d=\deg_n(p(n))$.
\end{lemma}

Note that $\deg_n(p(n))$ may not be equal to $\deg (p(n))$. For
example, $n^2-n+3$ is represented as $c_1(n)n + c_0(n)$ with $d=1$,
$c_1(n) =n-1$, and $c_0(n) = 3$.

\begin{proof}
  Let $p(x)=a_\ell x^\ell + \dots + a_1x + a_0$ with integral
  coefficients and $a_\ell > 0$. If the $a_i$ are all nonnegative, the
  proposition holds for $n>\max_{0\le i\le \ell}\{a_i\}$, otherwise we
  can prove it by induction on $\ell-i$, where $i$ is the smallest
  index such that $a_i < 0$. Suppose $i_0 \in \{0,\dots,\ell-1\}$ is
  the minimal number such that $a_{i_0} < 0$, then in the representation
  of $p(n)$, put $c_{j} = a_j$ for $j< i_0$, $c_{i_0} = n-a_{i_0}$. By
  induction, we have a representation for
  \[
  p(n) - (n+a_{i_0})n^{i_0} - a_{i_0-1}n^{i_0-1} - \cdots - a_0,
  \]
  so we can add it to $(n+a_{i_0})n^{i_0} + \cdots + a_0$ to get the
  desired representation for $p(n)$. Since for $n\gg 0$ we have
  $c_{i_0}=n+a_{i_0}>0$, by induction, we can make sure $0\le c_i\le
  n-1$ for $n\gg 0$. For a lower bound, $C=\max_{0\le i\le \ell}\{
  |a_i| \}+1$ is sufficient, i.e., for all $n\ge C$, the desired unique
  representation is guaranteed to exist. Note that $i_0\neq \ell$, since
  $a_{\ell}>0$. So this process will stop in finitely many
  steps. Uniqueness of this representation is clear.
\end{proof}

\begin{lemma}\label{reduce} Fix an integer $N$ and $n \gg 0$. Consider
  the set
\[
S_1(n) = \{(x_1,\dots,x_k)\in \Z^k \mid 0 \le x_i < n^{N+1},\
a_1(n)x_1+a_2(n)x_2+\cdots+a_k(n)x_k=m(n)\}
\]
where $a_i(n)=\sum^{d_i}_{\ell=0}a_{i\ell}n^{\ell}$ (as a usual
polynomial) and $m(n)=\sum_{\ell=0}^{d}b_{\ell}n^{\ell}$ (represented
in base $n$ as in Lemma \ref{repr}), with $b_d \ne 0$ and
$d\ge\max_{1\le i\le k}\{d_i\}$. Then $S_1(n)$ is in bijection with a
finite union of sets of the form
\[
S_2(n) = \{0 \le (x_{ij})_{\substack{1\le i\le k \\ 0\le j\le N}} < n,
x_{ij} \in \Z \mid \text{ all constraints on } x=(x_{ij}) \text{ are
  of the form }An+B = f(x)\}
\]
where $A,B\in\Z$ and $f(x)$ is a linear form of $x$ with constant
coefficients.
\end{lemma}

\begin{proof} 
  Fix a sufficiently large positive integer $n$. By our assumptions,
  for any point $(x_1, \dots, x_k) \in S_1(n)$, we can write $x_i =
  x_{i,N} n^N + x_{i,N-1} n^{N-1} + \cdots + x_{i,0}$ with $0\le
  x_{i,j} <n$.
  %
  %
  The rest of the lemma is a direct ``base $n$'' comparison starting
  from the lowest power to the highest power in the equation $\sum
  a_i(n) x_i = m(n)$.  To get a feel for the proof, we recommend that
  the reader look at Example~\ref{oneeq} first. We will use this as a
  running example to explain the steps of the proof.

  First write $m(n) = m_N(n) n^N + m_{N-1}(n) n^{N-1} + \cdots +
  m_0(n)$ in base $n$, so that each $m_i(n) = m_i''n + m'_i$ is a
  linear function in $n$. Then we know that $0 \le m_i(n) < n$ since
  we have fixed $n$ sufficiently large.  Going back to
  Example~\ref{oneeq}, we have $m(n) = 4n^2 + 3n - 5$ so that $m_2(n)
  = 4$, $m_1(n) = 2$ and $m_0(n) = n-5$, so that here ``sufficiently
  large'' means $n \ge 5$. Now expand out the equation $\sum a_i(n)
  x_i = m(n)$ to get
  \begin{align} \label{eqn:expand} \sum_{i=1}^k \left( (\sum_{\ell =
        0}^{d_i} a_{i\ell} n^\ell) (x_{i,N} n^N + \cdots + x_{i,0})
    \right) = m_N(n) n^N + \cdots + m_0(n).
  \end{align}
  The constant term in base $n$ of \eqref{eqn:expand} gives us the
  equation
  \[
  a_{10} x_{1,0} + \cdots + a_{k0} x_{k,0} = m_0(n) \pmod n.
  \]
  Since we have the bound $0 \le x_{i,0} < n$ for all $i$, we can in
  fact say that the LHS is equal to $m_0(n) + C_0n$ where $C_0$ is an
  integer such that $C_0$ is strictly greater than the sum of the
  negative $a_{i0}$ and strictly less than the sum of the positive
  $a_{i0}$ (and $C_0$ can also be 0). But note that $C_0$ only depends
  on the $a_{i0}$, so if our $n$ is sufficiently large, we may assume
  that $|C_0| < n$. Going back to Example~\ref{oneeq}, we have the
  equation
  \[
  2x_{10} + x_{20} = n-5 \pmod n,
  \]
  so that $0 \le C_0 \le 2$.

  So now we have finitely many cases for the value of $C_0$ to deal
  with. Fix one. Going back to the equation \eqref{eqn:expand}, we can
  substitute our value of $C_0$ and compare linear coefficients to get
  \[
  \sum_{i=1}^k (a_{i1}x_{i,0} + a_{i0}x_{i,1}) + C_0 = m_1(n) \pmod n.
  \]
  Again, since we know that $0 \le x_{i,0} < n$ and $0 \le x_{i,1} <
  n$ for all $i$, we can say that the LHS is equal to $m_1(n) + C_1n$
  where $C_1$ is an integer such that $C_1$ is greater than the sum of
  the negative $a_{i0}$ and $a_{i1}$ and less than the sum of the
  positive $a_{i0}$ and $a_{i1}$ (and this is independent of $C_0$
  because $|C_0| < n$). Going back to Example~\ref{oneeq}, we get the
  equation
  \[
  2x_{11} + x_{21} + x_{20} + C_0 = 2 \pmod n,
  \]
  so $0 \le C_1 \le 4$.
  
  Now we have finitely many cases of $C_1$, and again we fix one. We
  continue on in this way to get $C_2, C_3, \dots, C_{N-1}$. At each
  point, we only had finitely many choices for the next $C_i$, so at
  the end, we only have finitely many sequences $C_\bullet = (C_0,
  \dots, C_{N-1})$. For a given $C_\bullet$, we see that the $x_{ij}$
  must satisfy finitely many equalities of the form $f(x) = An + B$
  where $f$ is a linear form in $x$ with constant coefficients and
  $A,B \in \Z$. In Example~\ref{oneeq}, for $C_\bullet = (0,2)$, we
  would have the equations
  \begin{align*}
    2x_{10} + x_{20} &= n-5,\\
    2x_{11} + x_{21} + x_{20} &= 2 + \ul{2}n,\\
    2x_{12} + x_{21} + x_{30} + \ul{2} &= 4,
  \end{align*}
  along with the inequalities $0 \le x_{ij} < n$, and we have
  underlined the places where $C_1$ appears. 

  Each such $C_\bullet$ gives us a set of the form $S_2(n)$, and to
  finish we take the union of these sets over all sequences $C_\bullet$.
\end{proof}

We will need one more prepatory lemma before doing the proof of
Theorem~\ref{ehrhartconj}. 

\begin{lemma} \label{linear}If $P(n)$ is a polytope defined by
  inequalities of the form $An+B \le f(x)$, where $A,B\in\Z$ and
  $f(x)$ is a linear form of $x$ with constant coefficients, then $L_P
  \in \QP$.
\end{lemma}

\begin{proof} We first use the fact that the combinatorics of $P(n)$
  stabilizes for sufficiently large $n$ (see Section~\ref{equi}), say
  that the polytope has $s$ vertices $v_1(n), \dots, v_s(n)$. Also,
  note that the coefficients of the vertices of the polytope are
  linear functions in $n$ with rational coefficients. Let $D$ be the
  least common multiple of the denominators that appear in all of the
  coordinates. Then by \cite[Theorem 18.4]{ipp}, the function
  sending $(v_1(n), \dots, v_s(n))$ to the number of integer points in
  $P(n)$ is a polynomial if we restrict to a specific congruence class
  of $n$ modulo $D$. We can compose each of these functions with the
  polynomial $n \mapsto (v_1(n), \dots, v_s(n))$ to conclude that $L_P
  \in \QP$.
\end{proof}

We can also give a self-contained proof of Lemma~\ref{linear} that
only uses Ehrhart's theorem (Theorem~\ref{ehrharttheorem}).

\begin{proof}[Alternate proof of Lemma~\ref{linear}] We will do
  induction on dimension first, the case of dimension 0 being
  trivial. An inequality of the form $An + B \le f(x)$ with $B \ne 0$
  will be called an inequality with constant term. Secondly, we will
  do induction on the number of inequalities with constant term.
  Theorem~\ref{ehrharttheorem} gives the desired result when there are
  no inequalities with constant term. In general, let $P(n)$ be a
  polytope defined by $An + B \le f(x)$ plus other relations of the
  same form, which we will call $R(n)$. Let $P'(n)$ be the polytope
  with relations $An \le f(x)$ and $R(n)$. If $B > 0$, then for
  $i=0,\dots, B-1$, let $P_i(n)$ be the polytope with relations $An +
  i = f(x)$ and $R(n)$. Then
  \[
  L_P(n) = L_{P'}(n) - \sum_{i=0}^{B-1} L_{P_i}(n),
  \]
  so to show $L_P \in \QP$, it suffices to show that $L_{P'} \in \QP$
  and $L_{P_i} \in \QP$ for $i=0,\dots,B-1$. Since $P'(n)$ has one
  less inequality with constant term, we know that $L_{P'} \in \QP$ by
  induction. As for $P_i(n)$, write $f(x) = c_1x_1 + \dots + c_kx_k$
  and choose $j$ such that $c_j \ne 0$. Then the equation $An + i =
  f(x)$ can be rewritten as
  \[
  x_j = c_j^{-1}(An+i - \sum_{\ell \ne j} c_\ell x_\ell).
  \]
  Making substitutions into the relations $R$, each relation still has
  the form $A'n + B' \le f'(x)$ (clear denominators if necessary to
  make sure all of the coefficients are integral), and we have
  eliminated the variable $x_j$.

  If $B < 0$, then for $i=1,\dots,-B$, let $P_i(n)$ be the polytope
  with relations $An + i = f(x)$ and $R(n)$. Then
  \[
  L_P(n) = L_{P'}(n) + \sum_{i=1}^{-B} L_{P_i}(n),
  \]
  and we proceed as before. Hence $L_{P_i} \in \QP$ by induction on
  dimension.
\end{proof}

\begin{proof}[Proof of Theorem~\ref{ehrhartconj}]
  Consider a system of linear Diophantine equations $A(n)x=b(n)$ with
  finitely many nonnegative integer solutions for each $n\gg 0$, as in
  Theorem \ref{ehrhartconj}. Let
  \[
  a_1(n)x_1+a_2(n)x_2+\cdots+a_k(n)x_k=m(n)
  \]
  be any equation from the system. The equations $A(n)x=b(n)$ define a
  bounded polytope whose vertices are given by rational functions of
  $n$ as discussed in Section~\ref{equi}, where the degrees of the
  rational functions (difference between the degree of the numerator
  and the denominator) are independent of $n$. We can express each
  point in this polytope as a convex combination of the
  vertices. Therefore, writing each variable $x_1,\dots,x_k$ in base
  $n$, we can find some positive integer $N$, which does not depend on
  $n$, such that the coordinates of each point in the polytope are
  less than $n^{N+1}$. Using Lemma~\ref{reduce}, we can reduce each
  equation from $A(n)x=b(n)$ into a new system like $S_2$, with more
  variables than $S_1$ but all restrictions are of the form $An+b\le
  f(x)$. Applying Lemma~\ref{linear} finishes the proof.
\end{proof}

\subsection{Examples.}

\begin{eg} We give an example for Lemma~\ref{linear}. For $n$ a
  positive integer, let $P(n)$ be the polygon defined by the
  inequalities $x \ge 0$, $y \ge 0$ and $-2x-y \ge -n-1$. Then $P'(n)$
  is defined by the inequalities $x \ge 0$, $y \ge 0$, and $2x+y \le
  n$, and $P_1(n)$ is defined by the inequalities $x \ge 0$, $y \ge
  0$, and $n+1 = 2x+y$. We can rewrite the equality as $y = n + 1 -
  2x$, and then the other inequalities become $x \ge 0$ and $n + 1 \ge
  2x$.

  We see that $P'(n)$ is the convex hull of the points $\{(0,0),
  (0,n), (n/2,0)\}$, while $P_1(n)$ is the interval $[0,(n+1)/2]$. The
  total number of integer points in $P'(n)$ and $P_1(n)$ is given by
  the quasipolynomial
  \[
  \#(P(n) \cap \Z^2) = \begin{cases} k^2 + 3k + 2 & \text{if } n = 2k\\
    k^2+4k+4 & \text{if } n = 2k+1
  \end{cases}.
  \]
  Its rational generating function is
  \[
  \sum_{n \ge 0} \#(P(n) \cap \Z^2) t^n = \frac{t^5 - 3t^3 + 4t +
    2}{(1-t^2)^3} = \frac{t^3-2t^2+2}{(1-t)^3(1+t)}.
  \]
\end{eg}

\begin{eg}\label{oneeq} We give an example of Lemma~\ref{reduce}. Consider
  nonnegative integer solutions
  for $$2x_1+(n+1)x_2+n^2x_3=4n^2+3n-5.$$ For any $n>5$, ${\rm RHS} =
  4n^2+2n+(n-5)$ is the expression in base $n$. Now consider the left
  hand side. Writing $x_1,x_2,x_3$ in base $n$, let
  $x_1=x_{12}n^2+x_{11}n+x_{10}$, $x_2=x_{21}n+x_{20}$ and
  $x_3=x_{30}$ with $0\le x_{ij}<n$. In this case, $N = 2$, but we
  have further restrictions on the degrees of $x_2$ and $x_3$ in base
  $n$ coming from the coefficients $a(n)$. Then we have
\[ {\rm LHS} =
(2x_{12}+x_{21}+x_{30})n^2+(2x_{11}+x_{21}+x_{20})n+(2x_{10}+x_{20}).
\]
Now we can write the left hand side in base $n$ with extra constraints
on $(x_{ij})$'s.

We start with comparing the coefficient of $n^0$ in both sides. We
have the following three cases:
$$A_{0}^0=\{2x_{10}+x_{20}=n-5\},$$
$$A_{1}^0=\{2x_{10}+x_{20}=(n-5)+n\},$$
and $$A_2^0=\{2x_{10}+x_{20}=(n-5)+2n\}.$$ So in the language of proof
of Lemma~\ref{reduce}, we have $0 \le C_0 \le 2$.  We next consider
$n^1$ term. If $x$ satisfies $A_{C_0}^0$ for $n^0$, $C_0 \in
\{0,1,2\}$, then the equation is reduced to
$$(2x_{12}+x_{21}+x_{30})n^2+(2x_{11}+x_{21}+x_{20}+C_0)n=4n^2+2n.$$
Now compare the $n^1$ terms. We have five cases for each $C_0 \in
\{0,1,2\}$:
$$A_{C_0,C_1}^1=\{2x_{11}+x_{21}+x_{20}+C_0=C_1n+2\},$$ where
$C_1\in \{0,1,2,3,4\}$.

Last, we compare the $n^2$ terms. Note that since we assume $n\gg 0$,
the $n^0$ term won't carry over to $n^2$ term, so the computation of
$n^2$ term only depends on the term $n^1$. If $x$ satisfies the
$C_1$th condition for $n^1$, the equation then
becomes $$(2x_{12}+x_{21}+x_{30}+C_1)n^2=4n^2.$$ So for each $0 \le
C_1 \le 4$, we have $$A_{C_1}^2=\{2x_{12}+x_{21}+x_{30}+C_1=4\}.$$
Overall, we have the set $$\{(x_1,x_2,x_3)\in\Z_{\ge 0}^3\mid
2x_1+(n+1)x_2+n^2x_3=4n^2+3n-5\}$$ is in bijection with the
set $$\{x=(x_{12},x_{11},x_{10},x_{21},x_{20},x_{30})\in \Z_{\ge 0}^6,
0\le x_{ij}<n\},$$ such that $x$ satisfies the equations in the sets
$$
\begin{pmatrix}
A_0^0 & A_1^0 & A_2^0
  \end{pmatrix}
\begin{pmatrix}
A_{00}^1 & A_{01}^1 & \cdots & A_{04}^1 \\
A_{10}^1 & A_{11}^1 & \cdots & A_{14}^1 \\
A_{20}^1 & A_{21}^1 & \cdots & A_{24}^1
                   \end{pmatrix}
\begin{pmatrix}
A_0^2 \\
A_1^2 \\
\vdots \\
A_4^2
\end{pmatrix}.$$
Here we borrow the notation of matrix multiplication $AB$ to represent
intersection of sets $A\cap B$ and matrix summation $A+B$ to represent
set union $A\cup B$. Note that here all constrains $A_i^j$ on
$x=(x_{12},x_{11},x_{10},x_{21},x_{20},x_{30})$ are in the form of
$An+B\le f(x)$, where $A,B\in\Z$ and $f(x)$ is a linear form of $x$
with constant coefficients.
\end{eg}

\section{Generalized division.} \label{generalizeddivision}

An \textbf{integer-valued polynomial}, or a \textbf{numerical polynomial} is a function $f \colon \Z \to \Z$ with a
  polynomial $g(x) \in \Q[x]$ satisfying $f(n) = g(n)$ for all $n\in
  \Z$. For any numerical
polynomial $f$ and a constant $c \in \Z$, it is not hard to see that $\lfloor
\frac{f(n)}{c} \rfloor \in \QP$. We want to generalize this to
$\left\lfloor\frac{f(n)}{g(n)} \right\rfloor$ for two numerical
polynomials $f$ and $g$. Based on classical division, we have for all
$n\in \Z$ that $f(n) = \left\lfloor\frac{f(n)}{g(n)} \right\rfloor
g(n) + \left\{ \frac{f(n)}{g(n)}\right\} g(n)$, with $0\le \left\{
  \frac{f(n)}{g(n)}\right\} g(n)$, we define two integer-valued functions
$p(n)=\left\lfloor\frac{f(n)}{g(n)} \right\rfloor$, called the
\textbf{quotient} of $f(n)$ and $g(n)$ and $r(n) = \left\{
  \frac{f(n)}{g(n)} \right\} g(n)$, called the \textbf{remainder} of
$f(n)$ and $g(n)$. We will show that $p,r\in \QP$ by the
 generalized division algorithm.

This generalized division algorithm is based on the
division for $f,g\in \Q[x]$, except for two modifications. First, in
order to have $0\le r(n)<g(n)$ for $n\gg 0$, we need to change the
remainder condition of division in $\Q[x]$, as in [Algorithm \ref{divisionalg}, lines
\ref{rem1}--\ref{rem2}]. For example, for $f(x)=x-3$ and $g(x)=x$, then
we should have $p(x)=0$, $r(x)=x-3$, but these do not satisfy
$\deg(r)<\deg(g)$ nor $r=0$ for division in $\Q[x]$. Secondly, in
order to keep $p(x),r(x)\in \Z[x]$, we break $x$ into finitely many
``branches'' whenever we have noninteger coefficients, as in [Algorithm \ref{divisionalg}, lines
\ref{break1}--\ref{break2}]. For example, take $f(x)=x^2+3x$ and
$g(x)=2x+1$. Trying to do division in $\Q[x]$, leading coefficient of
$p$ will not be an integer, so we break $x$ into 2 ``branches'', since 2
is the leading coefficient of $g(x)$. Now we do division to
$f(2t)=4t^2+6t$ and $g(2t)=4t+1$, when $x=2t$; and
$f(2t+1)=4t^2+10t+4$ and $g(2t+1)=4t+3$, when $x=2t+1$. For $x=2t$, we
have $p_0(t)=t+1$ and $r_0(t)=t-1$ and for $x=2t+1$ we have
$p_1(t)=t+1$ and $r_1(t)=3t+1$. We will show (see Theorem~\ref{al1})
that breaking $x$ into $T=b_{\ell}^{k-\ell}$ (using notation in Algorithm \ref{divisionalg}) ``branches'' is sufficient to conclude that
$p_i(t),r_i(t) \in \Z[t]$, for $i=0,\dots,T-1$.

\begin{algorithm}
  \caption{Generalized division for $\Z[x]$.}
  \label{divisionalg}
  \begin{algorithmic}[1]
    \REQUIRE $f(x),g(x)\in \Z[x]$.
    \begin{align*}
      f(x) &=a_kx^k+a_{k-1}x^{k-1}+\cdots +a_1x+a_0\\
      g(x) &=b_\ell x^\ell + b_{\ell-1} x^{\ell-1} + \cdots +b_1x +
      b_0
    \end{align*}
    \ENSURE $T$ and $p_i(x)$, $r_i(x)\in \Z[x]$, for $i=0,\dots, T-1$
    such that
    \[
    p_i(m) = \left\lfloor \frac{f(n)}{g(n)} \right\rfloor \text{ and }
    r_i(m)=\left\lbrace \frac{f(n)}{g(n)} \right\rbrace g(n), \text{
      where } n = Tm+i \text{ and }n\gg 0
    \]
    \IF{$k<\ell$} \STATE output $T=1$
    \IF{$a_k>0$} \STATE $p(x)=0$; $r(x)=f(x)$.
    \ELSIF{$a_k<0$ \text{ and } $b_\ell < 0$} \STATE $p(x) = 1$;
    $r(x) = f(x) - g(x)$.
    \ELSIF{$a_k<0$ \text{ and } $b_\ell > 0$} \STATE $p(x)=-1$;
    $r(x) = f(x) + g(x)$.
    \ENDIF

    \ELSE \STATE output $T=b_\ell^{k-\ell}$.\label{break1}

    \FORALL{$i$ from 0 to $T-1$}
    \STATE Set $r_i(x)=f(Tx+i)$; $p_i(x)=0$; $K=\deg(r_i)(=k)$; $a=\lc(r_i)(=a_kT^k)$.\label{start}

    \WHILE{$K>\ell$} \STATE $h(x) = \frac{a}{b_\ell T^{\ell}}x^{K-\ell}$;
    \STATE $r_i(x)=r_i(x)-h(x)g(Tx+i)$;\label{degr-1}
    \STATE $p_i(x)=p_i(x)+h(x)$; \label{pi1}
    \STATE $K=\deg(r_i)$;
    \STATE $a=\lc(r_i)$
    \ENDWHILE\label{break2}

    \WHILE{$K=\ell$}
    \STATE $c= \left\lfloor \frac{a}{b_\ell T^{\ell}} \right\rfloor$;
    \STATE $p_i(x)=p_i(x)+c$; \label{pi2}
    \STATE $r_i(x)=r_i(x)-cg(Tx+i)$\label{ri2}
    \STATE $a=\lc(r_i)$
    \STATE $K=\deg(r_i)$
    \ENDWHILE

    \IF{$a>0$}\label{rem1}
    \STATE output $r_i(x)$, $p_i(x)$;
    \ELSIF{$a<0$}
    \STATE $r_i(x)=r_i(x)+g(Tx+i)$;
    \STATE $p_i(x)=p_i(x)-1$
    \ENDIF\label{rem2}
    \ENDFOR
    \ENDIF
  \end{algorithmic}
\end{algorithm}
Note that in Algorithm~\ref{divisionalg}, we define $\lc(r)$ to be
the leading coefficient of $r$.
\begin{theorem}\label{al1} Algorithm~\ref{divisionalg} is correct.
\end{theorem}

\begin{proof}
  The algorithm will stop after finitely many steps ($\deg(r)$ is
  reduced by 1 each step in [line \ref{degr-1}]). For $i=0,\dots, T-1$, and for $n\equiv i
  \pmod T$, we claim the following two properties:
  \begin{enumerate}
  \item \label{1}$r_i(x), p_i(x)\in \Z[x]$.
  \item \label{2}$f(Tx+i) = p_i(x)g(Tx+i) + r_i(x)$ and $0\le r_i(m)<g(n)$, for $n\gg 0$ and $n=Tm+i$.
\end{enumerate}
These two properties imply the correctness of the algorithm. By the algorithm, \ref{2} is obvious, so we only need to show \ref{1}. Note that in each step, we add to $p_i(x)$ some $h(x) = \frac{a}{b_\ell T^{\ell}}x^{K-\ell}$ or an integer $c$ (see [lines \ref{pi1} and \ref{pi2}]), and subtract $h(x)g(Tx+i)$ or $cg(Tx+i)$ from $r_i(x)$ (see [lines \ref{degr-1} and \ref{ri2}]), so it suffices to show $h(x)\in\Z[x]$. For every $r_i(x)$  with degree $K$ in [lines \ref{break1} - \ref{break2}], denote it by $r_i^K(x)$ and its leading coefficient by $a_K$. So we need to prove that $b_{\ell}T^{\ell}\mid a_K$ for all $K>\ell$. To show this, we use induction on $s=k-K$ to prove a stronger statement that $T^d\mid \text{coeff }(x^d)$ for every monomial $x^d$ in $r_i^K(x)$. If $s=0$, as in [line \ref{start}], $r_i^K(x)=f(Tx+i)$ and the claim holds by binomial expansion. Assume it still holds for some $s$ such that $K=k-s>\ell +1$. Then for $s+1$, we have $K=k-s-1$. In [line \ref{degr-1}], $$r_i^K(x)=r_i^{K+1}(x)-h(x)g(Tx+i),$$  so we only need to show that $$h(x)g(Tx+i)=\frac{a_{K+1}}{b_{\ell}T^{\ell}}x^{k-s-1-\ell}g(Tx+i)$$ also has the property that $T^d\mid\text{coeff }(x^d)$, for every monomial in $h(x)g(Tx+i)$. This is true because $T^d\mid \text{coeff }(x^d)$ in $g(Tx+i)$, and in $h(n)$ we have $T^{k-s}\mid a_{k-s}$ by hypothesis and $b_{\ell}T^{\ell}\mid T^{\ell+1}$ for $b_\ell\mid T$, and thus $T^{k-s-(\ell +1)}\mid \frac{a_{k-s}}{b_{\ell}T^{\ell}}$.
\end{proof}

We want to generalize Algorithm~\ref{divisionalg} to $\QP$. First we
need another expression for numerical polynomials.
\begin{lemma} \label{numericalpoly}
  Let $f$ be a numerical polynomial, i.e., $f \colon \Z \to \Z$ with a
  polynomial $g(x) \in \Q[x]$ satisfying $f(n) = g(n)$ for all $n\in
  \Z$. Then there exists $T$, and $g_i(x)\in\Z[x]$ for $i=0,\dots,T-1$
  such that $f(n)=g_i(m)$ for all $n=Tm+i$.
\end{lemma}

\begin{proof}
  Suppose $Tg(x) \in \Z[x]$ for some integer $T$. For $i=0,\dots,T-1$,
  define $g_i(x) = g(Tx+i)$. By the binomial expansion of $g(Tx+i) -
  g(i)$, it is easy to see that this is a polynomial with integer
  coefficients.
\end{proof}

Similarly, for an element $f \in \QP$, there exists $T \in \N$ and
$h_i(x)\in \Z[x]$ ($i=1,\dots,T$), satisfying $f(n)=h_i(m)$ for
$n\gg 0$ and $n= Tm+i$. We call $f=(T,\{h_i(x)\}_{i=1}^{T})$ a {\bf
  representation} of $f\in \QP$. Another simple property is that for
$f=(T_1,\{f_i(x)\}_{i=1}^{T_1})$ and $g=(T_2,\{g_i(x)\}_{i=1}^{T_2})$,
we can extend to a common period $T = \lcm (T_1,T_2)$ and
$f=(T,\{f'_i(x)\}_{i=1}^{T})$ and $g=(T,\{g'_i(x)\}_{i=1}^{T})$.

\begin{algorithm}
  \caption{Generalized division over $\QP$.}
  \label{QPdivision}
  \begin{algorithmic}[1]
    \REQUIRE $f,g\in\QP$
    \ENSURE $T$ and $p_i(x)$, $r_i(x) \in \Z[x]$,
    for $i=0,\dots, T-1$ such that
    \[
    p_i(m) = \left\lfloor\frac{f(n)}{g(n)} \right\rfloor \text{ and }
    r_i(m)= \left\lbrace \frac{f(n)}{g(n)} \right\rbrace g(n)\,\,\text{ where }n =
    Tm+i \text{ and } n\gg 0
    \]
    \STATE Find a common period $T_0$ for $f,g$, such that
    $f=(T_0,\{f_i(x)\}_{i=1}^{T_0})$ and
    $g=(T_0,\{g_i(x)\}_{i=1}^{T_0})$.

    \FOR{$i$ from $0$ to $T_0-1$}
    \IF{$g_i(x)=0$}
    \STATE Stop.
    \ELSE \STATE Apply Algorithm 1 to
    $f_i(x)$ and $g_i(x)$, get $T_i$, $p'_{i_j}(x), r'_{i_j}(x)\in
    \Z[x]$, $j=0,\dots,T_i-1$.
    \STATE Set $T = T_0\lcm(T_{i} \mid i=0,\dots,T_0-1)$, and get
    $p_i(x)$, $r_i(x)\in \Z[x]$, for $i=0,\dots, T-1$ from $p'_{i_j}(x),
    r'_{i_j}(x)$.
    \ENDIF
    \ENDFOR
  \end{algorithmic}
\end{algorithm}

\begin{theorem}[Generalized GCD] \label{gcd}
  Let $f,g$ be numerical polynomials, or more generally $f,g\in \QP$.
  Define $d(n) = \gcd(f(n), g(n))$. Then by the  generalized
  division over $\QP $ (Algorithm \ref{QPdivision}), we have $d \in \QP$. We call $d$ the {\bf
    generalized gcd} of $f$ and $g$, and denote it by $\ggcd(f,g)$.
\end{theorem}

\begin{proof} Apply Algorithm \ref{QPdivision} successively. Define
  $r_{-2}$ to be some constituent polynomial of $f$ and $r_{-1}$ to be
  some constituent polynomial of $g$ (a {\bf constituent polynomial}
  of a quasi-polynomial is a polynomial on a congruence class modulo
  its period). In general, if $r_{i-1} \ne 0$, define $r_i$ to be some
  constituent polynomial of $\text{rem}(r_{i-2}, r_{i-1})$. We will
  prove that $r_i = 0$ for some $i$ by induction on $\deg(r)$ and
  $\lc(r)$ (Recall that $\lc(r)$ means the leading coefficient of
  $r$.). Hence, since our choice of constituent polynomial was
  arbitrary, we will eventually have a remainder of 0 after finitely
  many steps. Hence the algorithm terminates after finitely many
  steps, or else there exists an infinite path in decision tree. If
  $\deg(r_i)=0$ and $\lc(r_i)>0$, the algorithm works the same as the
  classical Euclidean algorithm for integers. If $\deg(r_i)>0$, then
  since $0 \le r_{i+1}(n) < r_i(n)$ for $n \gg 0$, we have three cases
  for the next step:
\begin{enumerate}
\item $\deg(r_{i+1}) < \deg(r_i)$
\item $\deg(r_{i+1}) = \deg(r_i)$ and $\lc(r_{i+1}) < \lc(r_i)$. After
  finitely many iterations, one will either have $\lc(r_{i+1}) =
  \lc(r_i)$, or $\deg(r_{i+1}) < \deg(r_i)$.
\item $\deg(r_{i+1}) = \deg(r_i)$ and $\lc(r_{i+1}) = \lc(r_i)$. In
  the next step, $\deg(r_{i+2}) < \deg(r_{i+1})$.
\end{enumerate}
So the algorithm terminates after finitely many steps.
\end{proof}
Tracing through a natural generalization of the Euclidean algorithm,
we can get $u, v\in \QP$ such that $uf+vg=\ggcd(f,g)$. More generally
for $f_1,\dots, f_k\in \QP $ we have
\begin{corollary}\label{ggcd}
  Let $f_1,\dots, f_k\in \QP $, $d=\ggcd(f_1,\dots, f_k)$. Then $d\in
  \QP$ and there exists $u_1,\dots,u_k\in \QP$ with
  $\ggcd(u_1,\dots,u_k)=1$ such that
  \[
  f_1u_1+\dots+f_ku_k=d.
  \]
\end{corollary}

\begin{remark} This implies, in particular, that every finitely
  generated ideal of $\QP$ is principal. It is peculiar that in fact
  $\QP$ is not a Noetherian ring. For example, take $\chi_d$ to be the
  function such that $\chi_d(x) = 1$ if $d | x$, and 0 otherwise. Then
  $I = (\chi_2, \chi_4, \chi_8, \dots, \chi_{2^n}, \dots )$ is not a
  finitely generated ideal.
\end{remark}

Recall that for any $M\in \Z^{k\times s}$, there exist unimodular matrices
$U,V$, i.e., determinant $\pm 1$, such that $UMV = D
=(\text{diag}(d_1, \dots, d_r,0\dots, 0)| \textbf{0})$ with
$d_i|d_{i+1}$ and $D$ is called the Smith normal form of $M$. By
generalized division over $\QP$ (Algorithm~\ref{QPdivision}) and
Corollary~\ref{ggcd}, we can generalize the classical Smith normal form
theorem as follows.

\begin{theorem}\label{smith} For any $M\in (\QP)^{k\times s}$, define a
  matrix function $D \colon \Z \to \Z^{k\times s}$ such that $D(n)$ is
  the Smith normal form of $M(n)$. Then we have $D\in (\QP)^{k\times
    s}$ and there exists $U\in (\QP)^{k\times k}$, $V\in
  (\QP)^{s\times s}$ such that $U(n),V(n)$ are unimodular for
  $n\gg 0$ and $UMV=D$. We call this matrix function $D$ the \textbf{generalized Smith normal form} of $M$.
\end{theorem}

Before the proof, let us see a simple example for generalized gcd and generalized Smith normal form.
\begin{eg}Let $f=2$, $g=x$. Denote $d=\ggcd(f,g)$, namely $d(n)=\gcd(2,n)$. \\
  Put $T=2$ and $d_0(x)=2$, $d_1(x)=1$, $u_0(x)=1$, $u_1(x)=-x$,
  $v_0(x)=0$, $v_1(x)=1$, we
  have
  \[
  d(n)=d_i(m)=u_i(m)f(n)+v_i(m)g(n) \text{ for all } n=Tm+i \text{ and
  } i=0,1.
  \]

  For $A=\begin{pmatrix} 2 & 3x+2\\ x & x^3+2x \\ \end{pmatrix}$, let
  $U(x)=[U_{0}(x),U_{1}(x)]$, where
  \[
  U_{0}(x)=\begin{pmatrix} 1 & 0\\ -x & 1 \end{pmatrix},
  \quad U_{1}(x)=\begin{pmatrix} -x & 1\\ -2x-1 &
      2 \end{pmatrix},
  \]
  and $V(x)=[V_{0}(x),V_{1}(x)]$, where
  \[
  V_{0}(x)=\begin{pmatrix} 1 & -(3x+1)\\ 0 & 1 \end{pmatrix}, \quad
  V_{1}(x)=\begin{pmatrix} 1 & -(8x^3+6x^2+5x+3)\\ 0 &
      1 \end{pmatrix},
  \]
  and $D(x)=[D_{0}(x),D_{1}(x)]$, where
  \[
  D_{0}(x)=\begin{pmatrix} 2 & 0\\ 0 & 8x^3-6x^2+2x \end{pmatrix},
  \quad D_{1}(x)=\begin{pmatrix} 1 & 0\\ 0 &
    16x^3+12x^2+4x+1 \end{pmatrix}.
  \]
  We have $U_i(m)A(n)V(m)=D_i(m)$ is the Smith normal form for $A(n)$
  for $n\gg0$ where $n=2m+i$ and $i=0,1$.
\end{eg}

The following proof is based on the proof for classical Smith normal
form in \cite{new}.

\begin{lemma} \label{row}Let
$\alpha_1,\alpha_2,\dots,\alpha_k\in\QP$ with
$\ggcd(\alpha_1,\alpha_2,\dots,\alpha_k)=d_k$. Then there is a
matrix $W_k\in (\QP)^{k\times k}$ with first row
$[\alpha_1,\alpha_2,\dots,\alpha_k]$ and determinant $d_k$.
\end{lemma}

\begin{proof} By induction on $k$. For $k=1$, trivial. Now suppose the
  lemma is true for $k=n-1$, where $n\ge 2$, and let
  $W_{k-1}\in(\QP)^{(k-1) \times (k-1)}$ be a matrix with first row
  $[\alpha_1,\alpha_2,\dots,\alpha_{k-1}]$ and determinant
  $d_{k-1}=\ggcd(\alpha_1,\alpha_2,\dots,\alpha_{k-1})$.  Since
\begin{align*}
d_k&=\ggcd(\alpha_1,\alpha_2,\dots,\alpha_{k})\\
&=\ggcd(\alpha_1,\alpha_2,\dots,\alpha_{k-1}),\alpha_k)\\
&=\ggcd(d_{k-1},\alpha_k),
\end{align*}
by Corollary~\ref{ggcd}, there exist $\rho,\sigma\in\QP$ such
that $\rho d_{k-1}-\sigma\alpha_k=d_k$. Put
\[
W_k=\begin{pmatrix}
  &  &  &  & \alpha_k \\
  & W_{k-1} &  &  & 0 \\
  &  &  &  & \vdots \\
  &  &  &  & 0 \\
  \frac{\alpha_{1}\sigma}{d_{k-1}} & \frac{\alpha_{2}\sigma}{d_{k-1}} & \cdots & \frac{\alpha_{k-1}\sigma}{d_{k-1}} & \rho \\
\end{pmatrix}.
\]
Then we have $W_k\in (\QP)^{k\times k}$ with first row
$[\alpha_1,\alpha_2,\dots,\alpha_{k}]$ and $\det(W_k)=\rho d_{k-1}-\sigma\alpha_k=d_k$.
\end{proof}

\begin{proof}[Proof of Theorem $\ref{smith}$] Write $A=(a_{ij})\neq 0
  \in (\QP)^{k\times s}$. If $a_{11} \ne 0$, then switch rows and
  columns so that this is true. Without loss of generality, we may
  assume that $a_{11}$ divides all elements in the column. If not, let
  $d = \ggcd(a_{11}, \dots, a_{k1})$. By Corollary $\ref{ggcd}$, there
  exist $u_1, \dots, u_k \in \QP$ with $\ggcd(u_1,\dots,u_k)=1$ such
  that $u_1a_{11} + \dots + u_ka_{k1} = d$. By Lemma $\ref{row}$,
  there exists a matrix $W\in (\QP)^{k\times k}$ with first row
  $[u_1,\dots,u_k]$ and determinant $1$. So we have $A'=(a'_{ij})=WA$
  such that $a'_{11}$ divides all elements in the column. Suppose
  $a_{11}$ divides all elements $a_{ij}$. First, by the same process
  as above, we can assume that $a_{11}$ divides all elements in the
  first row and column. By elementary row and column operations, all
  the elements in the first row and column, other than the $(1,1)$
  position, can be made zero. If there is some element $a_{ij}$,
  $i,j\neq 1$ is not divisibleby $a_{11}$, add column $j$ to column
  $1$ and repeat the previous step to replace $a_{11}$ with one of its
  divisor which divides $a_{ij}$. Thus we must finally reach the stage
  where the element in the $(1,1)$ position divides every element of
  the matrix and all the other elements of the first row and column
  are zero.  Repeat the entire process with the submatrix obtained by
  deleting the first row and column.
\end{proof}

\section{Application of generalized division in the computation of
  some generalized Ehrhart polynomials.} \label{section:application}
We can see that the algorithm we give in the proof of
Theorem~\ref{ehrhartconj} is not very efficient.  In this section, by
the theory of generalized division developed in the previous section,
we are able to compute some special cases more efficiently.

\subsection{Generalized Popoviciu formula.} \label{section:example1}
First recall the classical Popoviciu formula. Let $a$ and $b$ be
relatively prime positive integers. Then the number of nonnegative
integer solutions $(x,y)$ to the equation $ax + by = n$ is given by
the formula \eqref{popoviciu}. Now consider the
equation $$a(n)x+b(n)y=m(n)$$ where $a(n),b(n),m(n)$ are integeral
polynomials of $n$. By the theory of generalized division, the above
formula for the number of nonnegative integer solutions
$p_{\{a_1(n),a_2(n)\}}(m(n))$ can be easily generalized to the
following

\begin{proposition}
  Suppose $\ggcd(a(n),b(n))=1$, then
$$p_{\{a_1(n),a_2(n)\}}(m(n)) = \frac{m(n)}{a_1(n)a_2(n)} -
\left\{\frac{a_1(n)^{-1}m(n)}{a_2(n)}\right\} -
\left\{\frac{a_2(n)^{-1}m(n)}{a_1(n)}\right\} + 1$$
where $a_1(n)a_1(n)^{-1}+a_2(n)a_2(n)^{-1}=1$. In particular,
$p_{\{a_1(n),a_2(n)\}}(m(n))\in \QP$.
\end{proposition}

If $\ggcd(a(n),b(n))=d(n)\neq 1$, we can simply divide it from both
sides and reduce to the case of $\ggcd(a(n),b(n))=1$.

Then as mentioned in the introduction, we can consider the number of
solutions $(x,y,z) \in \Z_{\ge 0}^3$ to the matrix equation
\eqref{matrixexample} where the $x_i$ and $y_i$ are fixed positive
integers and $x_{i+1}y_i < x_iy_{i+1}$ for $i=1,2$. Write $Y_{ij} =
x_iy_j - x_jy_i$. We assume that $\gcd(Y_{12}, Y_{13}, Y_{23}) = 1$,
so that there exist integers (not unique) $f_{ij}, g_{ij}$ such that
\[
\gcd(f_{12}Y_{13} + g_{12}Y_{23}, Y_{12}) = 1, \quad \gcd(f_{13}Y_{12}
+ g_{13}Y_{23}, Y_{13}) = 1, \quad \gcd(f_{23}Y_{13} + g_{23}Y_{12},
Y_{23}) = 1.
\]
Now define two regions $\Omega_i = \{(x,y) \mid \frac{y_i}{x_i} <
\frac{y}{x} < \frac{y_{i+1}}{x_{i+1}} \}$ for $i=1,2$. Then if $m =
(m_1, m_2) \in \Z^2$ is in the positive span of the columns of the
matrix in \eqref{matrixexample}, there exist the following
Popoviciu-like formulas \cite[Theorem 4.3]{xu} for $t(m|A)$, the
number of solutions of \eqref{matrixexample}.

When $m=(m_1,m_2)^{T}\in \overline{\Omega}_1\cap \Z^{2}$,
\begin{align*}
  t(m|A)=& \frac{m_2x_1-m_1y_1}{Y_{12}Y_{13}}-
  \left\{\frac{(f_{12}Y_{13}+g_{12}Y_{23})^{-1}
      (m_2(f_{12}x_1+g_{12}x_2)-m_1(f_{12}y_1+g_{12}y_2))}{Y_{12}}\right\} 
  \\
  &-\left\{\frac{(f_{13}Y_{12}+g_{13}Y_{23})^{-1}
      (m_2(f_{13}x_1+g_{13}x_3)-m_1(f_{13}y_1+g_{13}y_3))}{Y_{13}}\right\}+1.
\end{align*}
When $m=(m_1,m_2)^{T}\in \overline{\Omega}_2\cap \Z^{2}$,
\begin{align*}
  t(m|A)=& \frac{m_1y_3-m_2y_3}{Y_{23}Y_{13}}-
  \left\{\frac{(f_{23}Y_{13}+g_{23}Y_{12})^{-1}
      (m_1(f_{23}x_3+g_{23}x_2)-m_2(f_{23}y_3+g_{23}y_2))}{Y_{23}}\right\}
  \\
  &-\left\{\frac{(f_{13}Y_{12}+g_{13}Y_{23})^{-1}
      (m_1(f_{13}x_1+g_{13}x_3)-m_2(f_{13}y_1+g_{13}y_3))}{Y_{13}}\right\}+1, 
\end{align*}
where $f_{12},g_{12},f_{13},g_{13},f_{23}$ and $g_{23}\in \Z$ satisfy
\[
\gcd(f_{12}Y_{13}+g_{12}Y_{23},Y_{12}) =
\gcd(f_{13}Y_{12}+g_{13}Y_{23},Y_{13})=
\gcd(f_{23}Y_{13}+g_{23}Y_{12},Y_{23})=1
\]
and
\begin{align*}
  (f_{12}Y_{13}+g_{12}Y_{23})^{-1}(f_{12}Y_{13}+g_{12}Y_{23}) &\equiv
  1 \pmod {Y_{12}},\\
  (f_{13}Y_{12}+g_{13}Y_{23})^{-1}(f_{13}Y_{12}+g_{13}Y_{23}) &\equiv
  1 \pmod{Y_{13}},\\
  (f_{23}Y_{13}+g_{23}Y_{12})^{-1}(f_{23}Y_{13} +g_{23}Y_{12}) &\equiv
  1 \pmod{Y_{23}}.
\end{align*}

Notice that in the above formula, everything can be directly
generalized to the ring $\Z[x]$ using generalized division and
generalized GCD. In other words, we can replace the $x_i$, $y_i$, and
$m_i$ by polynomials in $n$ in such a way that for all values of $n$,
the condition $\gcd(Y_{12}, Y_{13}, Y_{23}) = 1$ always holds. For
example, consider the system
\[
\left( \begin{matrix} 2n+1 & 3n+1 & n^2 \\ 2 & 3 & n+1 \end{matrix}
\right) \left( \begin{matrix} x \\ y \\ z \end{matrix} \right) =
\left( \begin{matrix} 3n^3 + 1 \\ 3n^2 + n - 1 \end{matrix} \right).
\]
Then for $n \gg 0$, we have that
\[
\frac{3}{3n+1} < \frac{3n^2 + n - 1}{3n^3 + 1} < \frac{n+1}{n^2},
\]
so the number of solutions $(x,y,z)$ is counted by the formula of $t(m|A)$ for $m\in \overline{\Omega}_2$ and it is easy to see that  $t(m|A)\in\QP$.

\subsection{Dimension two.}
For polytopes in dimension two, there is an efficient algorithm for
computing the number of lattice points using continued fractions. See
\cite[Chapter 15]{ipp} for the algorithm in the case of
non-parametrized two dimensional polytopes.

By the generalized division theory, we can generalize this algorithm
to the case of a polytope whose vertices are given by
polynomials. Notice that the only difference is the generalization of
continued fractions: for any $f,g\in \Z[x]$, we expand $f/g$ as a
continued fraction. Then by Theorem~\ref{gcd}, for $n\gg 0$, the terms
in the expansion are quasi-polynomials and the number of terms does not
depend on $n$.

We use the following notation for continued fractions. First, set $[a]
= a$ and in general, we set
\[
[a_0; a_1, \dots, a_k] = a_0 + \frac{1}{[a_1; a_2, \dots, a_k]}.
\]

\begin{example}
 For $n>4$, we have
$$\frac{n^2}{2n+1}=\begin{cases}
  [m-1;\,1,\,3,\,m] & n=2m\\
  [m-1;\,3,\,1,\,m-1] & n=2m-1
  \end{cases}.
$$And for the cone generated by $(0,1)$ and $(2n+1,n^2)$, we have the following efficient algorithm of decomposition into prime cones.

$
\xy
0;/r.2pc/:
(40,0)*{}="G"; (50,0)*{}="H"; (50,10)*{}="L";
(40,-3)*{\textbf{0}};
(50,-3)*{\textbf{1}};
(50,14)*{(2n+1,n^2)};
 "G"; "L"**\dir{-};
 "H"; "G"**\dir{-};
\endxy
=
\xy
0;/r.2pc/:
(40,0)*{}="G"; (50,0)*{}="H"; (40,10)*{}="L";
(40,-3)*{\textbf{0}};
(50,-3)*{\textbf{1}};
(37,10)*{\textbf{1}};
 "G"; "L"**\dir{-};
 "H"; "G"**\dir{-};
\endxy
-
\xy
0;/r.2pc/:
(40,0)*{}="G"; (50,6)*{}="H"; (40,10)*{}="L";
(40,-3)*{\textbf{0}};
(52,10)*{(1,m-1)};
(37,10)*{\textbf{1}};
 "G"; "L"**\dir{-};
 "H"; "G"**\dir{-};
\endxy
+ 
\xy
0;/r.2pc/:
(40,0)*{}="G"; (50,6)*{}="H"; (47,12)*{}="L";
(40,-3)*{\textbf{0}};
(58,10)*{(1,m-1)};
(50,15)*{(1,m)};
 "G"; "L"**\dir{-};
 "H"; "G"**\dir{-};
\endxy
 - 
\xy
0;/r.2pc/:
(40,0)*{}="G"; (53,10)*{}="H"; (47,12)*{}="L";
(40,-3)*{\textbf{0}};
(64,5)*{(4,4m-1)};
(50,15)*{(1,m)};
 "G"; "L"**\dir{-};
 "H"; "G"**\dir{-};
\endxy
+
\xy
0;/r.2pc/:
(40,0)*{}="G"; (53,10)*{}="H"; (50,10)*{}="L";
(40,-3)*{\textbf{0}};
(64,5)*{(4,4m-1)};
(60,14)*{(2n+1,n^2)};
 "G"; "L"**\dir{-};
 "H"; "G"**\dir{-};
\endxy$
\end{example}

\bigskip

\filbreak \noindent Sheng Chen, 
Department of Mathematics, 
Harbin Institute of Technology, 
Harbin, China 150001, 
{\tt schen@hit.edu.cn}

\bigskip

\filbreak \noindent Nan Li and Steven V Sam, 
Department of Mathematics, 
Massachusetts Institute of Technology, 
Cambridge, MA 02139, 
{\tt \{nan,ssam\}@math.mit.edu}
\end{document}